\documentclass[12pt]{amsart}
\usepackage{amsmath,enumitem}
\usepackage[latin1]{inputenc}
\usepackage{amsmath}
\usepackage{amsthm,amstext,amsfonts,bm,amssymb}
\usepackage{graphics,graphicx}
\usepackage{array}
\usepackage{pdflscape}
\usepackage{a4wide}
\usepackage{color, url}
\usepackage{float}

\theoremstyle{plain}
\newtheorem{theorem}{Theorem}

\theoremstyle{definition}
\newtheorem{definition}[theorem]{Definition}

\newtheorem{remark}[theorem]{Remark}
\newtheorem{question}[theorem]{Question}

\usepackage{arydshln}

\newcommand{\Nu}{{\mathcal N}}

\newcommand{\Z}{{\mathbb Z}}

\newcommand{\Lo}{{\mathcal L}}
\def\Inv{\text{\rm inv}}

\def\li{\text{\rm Li}}

\def\stirling#1#2{\genfrac{\{}{\}}{0pt}{}{#1}{#2}}
\title[On $q$-poly-Bernoulli numbers arising from combinatorial interpretations]{On $q$-poly-Bernoulli numbers arising from combinatorial interpretations}

\author{Be\'ata B\'enyi}
\address{\noindent Faculty of Water Sciences, National University of Public Service, Budapest, HUNGARY}
\email{beata.benyi@gmail.com}

\author{Jos\'e L. Ram\'{\i}rez}
\address{\noindent Departamento de Matem\'aticas, Universidad Nacional de Colombia, Bogot\'a,  COLOMBIA}
\email{jlramirezr@unal.edu.co}

\date{\today}

\begin{document}
\begin{abstract}
In this paper we present several natural $q$-analogues of the poly-Bernoulli numbers arising in combinatorial contexts. We also recall some relating analytical results and ask for combinatorial interpretations.
\end{abstract}

\maketitle

\section{Introduction}
Poly-Bernoulli numbers were introduced by Kaneko in 1997 \cite{Kaneko} as he noticed that the generating function of the usual Bernoulli numbers  can be generalized in a nice way using the polylogarithm function.  Later, in 1999  Arakawa and Kaneko  \cite{AraKan} observed that  poly-Bernoulli numbers can be expressed as special values of multiple zeta values.  Motivated by this observation, they introduced a function, nowadays known as the Arakawa-Kaneko function, which expresses special values of this function at negative integers with the help of poly-Bernoulli numbers and multiple zeta values.  Since then poly-Bernoulli numbers were studied by numerous authors.

The literature in this topic is so wealth that we can not give here a complete list about all the areas that was motivated by these works, however, we mention a few. Generalizations were introduced, as for instance multi-poly-Bernoulli numbers \cite{Bayad} or poly-Bernoulli numbers associated with a Dirichlet character \cite{Bayad2}. Analogue numbers were also introduced as poly-Euler \cite{Ohno}, poly-Cauchy \cite{Komatsu}, poly-Eulerian numbers \cite{Son} were introduced. Beyond the numbers, polynomials (and generalizations of the polynomials) were defined and studied. In another direction, using the formula of poly-Bernoulli numbers, that involves the Stirling numbers of the second kind, the authors replaced in the formula variations of the Stirling numbers and studied the so obtained number sequences and polynomials, respectively. For example, the classical Stirling numbers are replaced by the incomplete Stirling numbers \cite{BenRam,KomLipMez}  or the $r$-Stirling numbers \cite{BenRam2,KomRam,}.

Poly-Bernoulli numbers received attention by combinatorialists because poly-Bernoulli numbers with negative $k$ indices enumerate various combinatorial objects and have very interesting combinatorial properties. We find also generalizations in the literature that are natural from the combinatorial point of view. One of the first combinatorial interpretations was given by Brewbaker \cite{Brewbaker} in terms of lonesum matrices.

In combinatorics, number theory, and theory of special polynomials the so called $q$-analogues of number sequences are often investigated. In combinatorial context $q$ is usually used for a parameter of the underlying combinatorial object which are enumerated by the number itself if the extra parameter is not taking into account (so for $q=1$). In this paper we focus on the aspects of $q$-analogues of the poly-Bernoulli numbers. There are different ways to attack this problem. The aim of this paper is to show some of these directions. Most of all we want to highlight the combinatorial richness of the theory of poly-Bernoulli numbers by defining several $q$-analogues of this counting sequence that arise naturally from the combinatorial interpretations.
We also point out a $q$-analogue which was defined analytically and pose the question of finding an appropriate combinatorial interpretation. Further, we also define a $q$-analogue that arise from the formula naturally, and show the strong connection to the Akiyama-Tanigawa algorithm.
Unfortunately, in this case we neither found yet any nice combinatorial interpretation.
We illustrated our results with examples in order to facilitate the reading of the paper and to show the wealth of the theory.

The outline of the paper is as follows. First, we recall some necessary definitions and notation we use throughout the paper. The definitions that are specific for the considered $q$-analogue in a certain section is given in the section itself. In Section 3, we give a natural combinatorial interpretation of the $q$-poly-Bernoulli numbers by using ordered partitions and the inversion statistic.  In Section 4, we use the weight defined by Cigler in his study. We interpret this weight on the set of lonesum matrices to give a new $q$-analogue. In Section 5, we exploit the interpretation of the poly-Bernoulli numbers in the context  of  Vesztergombi permutations. Then, by using the $q$-rook numbers  studied by Garsia and Remmel,  we give a combinatorial identity to find these new polynomials. In Section 6, we reveal a connection between the poly-Bernoulli numbers and the theory of PASEP, then we introduce a new possible $q$-analogue in this context. In Section 7, we show some possible relation between the $q$-analogue defined by  Cenkci and Komatsu and the $\Gamma-$free matrices.  Finally, we use the Akiyama-Tanigawa algorithm to give a non-combinatorial definition of $q$-poly-Bernoulli numbers.

\section{Definition and notation}

The \emph{poly--Bernoulli numbers},  denoted by $B_{n}^{(k)}$, where $n$ is a positive integer and $k$ is an
integer, are defined by the following exponential generating function
\begin{equation}\label{gefunpoly}
\sum_{n=0}^{\infty}B_n^{(k)}\frac{x^n}{n!} =\frac{\li_k(1-e^{-x})}{1-e^{-x}},
\end{equation} where
\[
\li_k(z) =\sum_{i=1}^{\infty}\frac{z^i}{i^k}
\]
is the $k$-th polylogarithm function. Note that for $k=1$ we recover the classical Bernoulli numbers, that is,  $B_n^{(1)}=(-1)^nB_n$, for $n\geq 0$, where $B_n$ denotes the $n$-th Bernoulli number.

From the combinatorial point of view the array with negative $k$ indices are interesting, since these numbers are integers (see sequence  A099594 in \cite{OEIS}).  The first few values of this array are 
$$\
\begin{array}{c|cccccc}
k$\textbackslash$ n & 0 & 1 & 2 & 3& 4& 5\\ \hline
0 & 1 & 1 & 1 & 1 & 1 & 1 \\
-1 & 1 & 2 & 4 & 8 & 16 & 32\\
-2 & 1 & 4 & 14 & 46 & 146 & 454\\
-3 & 1 & 8 & 46 & 230 & 1066 & 4718\\
-4 & 1 & 16 & 146 & 1066 & 6902 & 41506\\
-5 & 1 & 32 & 454 & 4718 & 41506 & 329462\\
\end{array}
$$
Several explicit formulas for the poly-Bernoulli numbers are known:
\begin{align}
B_{n}^{(k)}&=(-1)^n\sum_{m=0}^n(-1)^m\frac{ m!\stirling{n}{m}}{(m+1)^k},  \label{eq:pBszita}\\
B_{n}^{(-k)}&= \sum_{m=0}^{\min(n,k)}m!\stirling{n+1}{m+1}m!\stirling{k+1}{m+1}, \label{eq:pBcomb}\\
B_{n}^{(-k-1)}&= B_{n}^{(-k)} + \sum_{m=1}^n\binom{n}{m}B_{n-(m-1)}^{(-k)}. \label{eq:rec}
\end{align}
The first combinatorial interpretation of poly-Bernoulli numbers are \emph{lonesum matrices} \cite{Brewbaker}.  Lonesum matrices are binary matrices that are uniquely reconstructible from their column and row sum vectors. They can be characterized by forbiddance of the submatrices of the form
$$\begin{pmatrix} 0 & 1 \\ 1 & 0 \end{pmatrix} \quad \text{or} \quad \begin{pmatrix} 1 & 0 \\ 0 & 1 \end{pmatrix}.$$

Similarly, there are three other pairs of $2\times 2$ submatrices such that their forbiddance leads to a class of $(0,1)$-matrices enumerated by the poly-Bernoulli numbers,  see for example \cite{BH1, BH}. The formula \eqref{eq:pBcomb} enumerates pairs of ordered partitions, which is an obvious interpretation and can be turned immediately to a so called \emph{Callan permutation} \cite{BH1}. Another permutation class which we call \emph{Vesztergombi permutations} was found early in an algebraic study \cite{Launois}. In this paper we will only refer to these interpretations though there are also graph theoretical applications (cf. \cite{BenRam2}), for instance.

In the formulas \eqref{eq:pBszita} and \eqref{eq:pBcomb} we see that the Stirling numbers of the second kind plays a key role. Stirling numbers of the second kind, $\stirling{n}{k}$,  are defined as the number of set partitions of an $n$-element set into $k$ non-empty blocks.  Stirling numbers of the second kind have several $q$-analogues introduced in the literature (cf \cite{Mansour1, Mansour2}). In this paper we use different definitions, appropriate for our purposes of the specific sections. We recall the particular $q$-analogue in the section itself.

As it is usual in the theory of $q$-calculus
\begin{align*}
[n]_q:=1+q+q^2+\cdots +q^{n-1}=\frac{1-q^n}{1-q}.
\end{align*}
We also use the notation $[n]!_q:=[1]_q[2]_q\cdots[n]_q$. If it is not confusing and not necessary to write explicitly out we leave the subscript $q$ and write simple $[n]$ instead of $[n]_q$. Similarly, $[n]!$ instead of $[n]!_q$. Note that the above expressions become the integer $n$ and the factorial $n!$ when $q=1$.

\section{$q$-analogue based on the obvious interpretation}
First, we generalize the formula \eqref{eq:pBcomb} in a straightforward manner. We remember the definition of Carlitz's $q$-Stirling numbers:
\begin{align}\label{qStir2}
\stirling{n}{m}_q=\stirling{n-1}{m-1}_q+[k]_q\stirling{n-1}{m}_q,
\end{align}
with the initial condition $\stirling{n}{0}_q=\stirling{0}{m}_q=0$ except $\stirling{0}{0}_q=1$.

We define the following $q$-analogue of the poly-Bernoulli numbers
\begin{align}\label{qnalog2}
B_{n,q}^{(k)}:=\sum_{m=0}^{\min(n,k)} [m]! {n+1 \brace m+1}_q[m]! {k+1 \brace m+1}_q.
\end{align}
Before we give the combinatorial interpretation of these numbers it is useful to consider the $q$-analogue of the Fubini numbers.

 The \textit{Fubini numbers}, also called the \textit{ordered Bell
numbers}, $F_n$, count the total number of set partitions of $\{1,2,\ldots, n\}$ where the blocks are ordered. They are given by
\begin{equation}
F_n=\sum_{k=0}^n k!{n \brace k} \text{ for }  n>0 \text{ and }  F_{0} =1.
\label{fubini-formula1}
\end{equation}

Note that the combinatorial identity \eqref{eq:pBcomb} of poly-Bernoulli numbers has a great similarity with \eqref{fubini-formula1} of Fubini numbers. For this reason, in this section we will study a $q$-analogue of the Fubini numbers. We present a combinatorial interpretation by using inversions.

\begin{definition}\cite[pp.205]{Mansour1}
Let $\omega= B_1/B_2/ \cdots /B_k$ be any set partition and $b\in B_i$.
We will say that $(b,B_j)$ is an \emph{inversion} if $b > \min B_j$ and $i < j$. We define
the \emph{inversion number} of $\omega$, written $\Inv^*(\omega)$, to be the number of inversions in $\omega$.
\end{definition}

For example, if $\omega=137/26/45$  then $(3,B_2), (7,B_2), (7,B_3), $ and $(6,B_3)$ are the inversions of $\omega$ and $\Inv^*(\omega)=4$.

Note that the  $q$-Stirling numbers defined by the recursion \eqref{qStir2} have the following combinatorial interpretation
$${n \brace k}_q=\sum_{\omega\in \Pi_{n,k}}q^{\Inv^*(\omega)},$$
where $\Pi_{n,k}$ is the number of set partitions of $\{1,2,\ldots, n\}$ into $k$ blocks.

 We define the \emph{$q$-Fubini numbers} using q-analogues of the expressions of \eqref{fubini-formula1} by the equality
$$F_{n,q}:=\sum_{k=0}^n[k]_q! {n \brace k}_q.$$
It is clear that $\lim_{q\to 1} F_{n,q}=F_n.$ The first few $q$-Fubini numbers are
$$1, \quad 1, \quad q+2, \quad q^3+3 q^2+5 q+4, \quad q^6+4 q^5+9 q^4+16 q^3+20 q^2+17 q+8, \dots$$

\begin{theorem}
The $q$-Fubini numbers are given by
 $$F_{n,q}=\sum_{\pi\in O_{n}}q^{\Inv^*(\pi)},$$
 where $O_n$ is the set of ordered set partitions of $\{1,2,\ldots,n\}$.
 \end{theorem}
 \begin{proof}
Let $T_q(n,k):=[k]_q! {n \brace k}_q$. From \eqref{qStir2} it is possible to prove that $T_q(n,k)$ satisfies the recurrence relation
\[
T_q(n,k)=[k]_q\left(T_q(n-1,k-1) + T_q(n-1,k)\right),
\]
with the initial condition $T_q(n,0)=0=T_q(0,k)$ for $n>0$ and  $T_q(0,0)=1$.  On the other hand, let $U_q(n,k):=\sum_{\pi\in O_{n,k}}q^{\Inv^*(\pi)},$ where $O_{n,k}$ is the set of ordered set partitions of $\{1,2,\ldots, n\}$ into $k$ non-empty blocks.  The sequence $U_q(n,k)$ satisfies the same recurrence relation of $T_q(n,k)$, with the same initial values. Indeed, for any ordered partition  $\pi$ of  $\{1,2,\ldots, n\}$ into $k$ non-empty blocks, there are two options: either the element $n$ is in a single block or it is in a block with more than one element. In the first case,  if the single block is in the $i$-th position, the contribution is $q^{k-i}U_q(n-1,k-1)$, for $i=1, 2, \dots, k$. Summing over the possible values of $i$, we obtain $[k]_qU_q(n-1,k-1)$.

In the second case, the element $n$ can be placed into one of the $k$ blocks. Then its contribution is $q^{k-i}U_q(n-1,k)$, where $i$ is the position of the block from left to right. Summing over the possible values of $i$, we obtain $[k]_qU_q(n-1,k)$.  Therefore,  we conclude that $U_q(n,k)=T_q(n,k)$ for all $n, k\geq 0$.
\end{proof}

For example,  $F_{3,q}=q^3+3 q^2+5 q+4,$ with the ordered partitions and their respective weights  being
\begin{align*}
&\{ 1, 2, 3\}   \rightarrow  1  & &  \{ 1, 2\}, \{ 3\}   \rightarrow  1 &  &  \{ 3\},  \{ 1, 2\}   \rightarrow  q  &   &  \{ 1,3\},  \{ 2\}   \rightarrow  q  & \\
&  \{ 2\},  \{ 1, 3\}   \rightarrow  q  & &  \{2, 3\},  \{ 1\}   \rightarrow  q^2  & &  \{ 1\},  \{ 2,3\}   \rightarrow  1  &   &  \{ 1\},  \{ 2\},  \{ 3\}   \rightarrow  1  & \\
  &  \{ 1\},  \{ 3\},  \{ 2\}   \rightarrow  q  &   &  \{ 2\},  \{ 1\},  \{ 3\}   \rightarrow  q  &    &  \{ 2\},  \{ 3\},  \{ 1\}   \rightarrow  q^2  &    &  \{ 3\},  \{ 2\},  \{ 1\}   \rightarrow  q^3  & \\
    &  \{ 3\},  \{ 1\},  \{ 2\}   \rightarrow  q^2  &
\end{align*}

Now we turn our attention to the combinatorial interpretation of the formula \eqref{qnalog2}.

The formula \eqref{eq:pBcomb} have the following obvious interpretation. Let $\widehat{N}$ the set $\{1,2,\ldots, n\}$ extended by a special element $\overline{0}$ and $\widehat{K}$ the set $\{1,2,\ldots, k\}$ extended by another special element $\overline{k+1}$. \eqref{eq:pBcomb} counts the alternating sequence of blocks of partitions of $\widehat{N}$ and $\widehat{K}$ such that the first block contains $\overline{0}$ and the last block contains $\overline{k+1}$ \cite{BH1}. Note that the number of blocks of $\widehat{N}$ and $\widehat{K}$ has to be equal to provide such an alternating sequence.
The interpretation of the $q$-analogue of poly-Bernoulli numbers $B_{n,q}^{(k)}$ should be clear based on this fact and the interpretation of the $q$-Fubini numbers. We formulate it in the next theorem precisely.

\begin{theorem}
Let $\widehat{N}=\{\overline{0}, 1,\ldots, n\}$ and $\widehat{K}=\{1,2,\ldots, k, \overline{k+1}\}$ be two sets. We let $\pi_{n,k}$ denote an alternating sequence of blocks of partitions of $\widehat{N}$ and $\widehat{K}$ such that the first block contains $\overline{0}$ and the last block contains $\overline{k+1}$. Further, we let $\mathcal{OP}_{n,k}$ denote the set of all such alternating sequences. The $q$-poly-Bernoulli numbers $B_{n,q}^{(k)}$ counts the inversions in the sequences $\pi_{n,k}$:
\begin{align*}
B_{n,q}^{(k)}=\sum_{\pi_{n,k}\in \mathcal{OP}_{n,k}} q^{\Inv^*(\pi_{n,k})}.
\end{align*}
\end{theorem}

For example, $B_{3,q}^{(1)}= 4 + 3 q + q^2$. This polynomial can be obtained by counting the inversions in the pairs of ordered partitions of the augmented sets $\{\textcolor{blue}{\overline{0}, 1, 2, 3}\}$ and $\{\textcolor{red}{1, \overline{2}}\}$. We list the elements of $\mathcal{OP}_{3,1}$ with their weights.
\begin{align*}
&\textcolor{blue}{\overline{0}123}/\textcolor{red}{1\overline{2}}  \rightarrow  1  \hspace{1cm} \textcolor{blue}{\overline{0}}/\textcolor{red}{1}/\textcolor{blue}{123}/\textcolor{red}{\overline{2}}  \rightarrow  1  \hspace{1cm}  \textcolor{blue}{\overline{0}1}/\textcolor{red}{1}/\textcolor{blue}{23}/\textcolor{red}{\overline{2}}  \rightarrow  1  \hspace{1cm}  \textcolor{blue}{\overline{0}23}/\textcolor{red}{1}/\textcolor{blue}{1}/\textcolor{red}{\overline{2}}  \rightarrow  q^2  \\
&  \textcolor{blue}{\overline{0}12}/\textcolor{red}{1}/\textcolor{blue}{3}/\textcolor{red}{\overline{2}}  \rightarrow  1 \hspace{1cm} \textcolor{blue}{\overline{0}3}/\textcolor{red}{1}/\textcolor{blue}{12}/\textcolor{red}{\overline{2}}  \rightarrow  q   \hspace{1cm}   \textcolor{blue}{\overline{0}13}/\textcolor{red}{1}/\textcolor{blue}{2}/\textcolor{red}{\overline{2}}  \rightarrow  q \hspace{1cm} \textcolor{blue}{\overline{0}2}/\textcolor{red}{1}/\textcolor{blue}{13}/\textcolor{red}{\overline{2}}  \rightarrow  q
\end{align*}

Notice that the number of blocks of the partitions of the two sets $\widehat{N}$ and $\widehat{K}$ is equal.

\section{$q$-analogue  on lonesum matrices}

In this section we introduce another $q$-analogue that arise naturally as a generalization of the formula \eqref{eq:pBcomb}. For this aim we recall a $q$-analogue of the Stirling numbers of the second kind defined by Cigler \cite{Cigler} based on a block of a set partition that contains a particular element.

Let us denote by $\Pi (n,k)$ the set of partitions of $\{0, 1, \dots, n-1\}$ into $k$ non-empty blocks. It is clear that $|\Pi(n,k)|={n \brace k}$. Suppose $\pi\in\Pi(n,k)$  is represented as $B_0/B_1/\cdots/ B_{k-1}$, where $B_0$ denotes the block containing the element zero. Define the \emph{weight $w_1$} of the partition $\pi$ by letting
$$w_1(\pi):=q^{\sum_{i\in B_0}i}.$$

The \emph{$q$-Stirling numbers of the second kind} \cite{Cigler},  denoted by ${n \brace k}^*_q$,  are defined by
$${n \brace k}^*_q=\sum_{\pi\in \Pi(n,k)}q^{w_1(\pi)}.$$

Based on this definition, we introduce the following $q$-analogue of poly-Bernoulli numbers. As we mentioned in the introduction a $01$ matrix is called \emph{lonesum} if it is uniquely reconstructible from its row and column sum vectors.

For example,  $B_2^{(-3)}=46$, the lonesum matrices of size $2\times 3$ appear in Figure \ref{fig1}. Note that we use  black squares for the ones and white squares for the zeros.
\begin{figure}[H]
\centering
\includegraphics[scale=1]{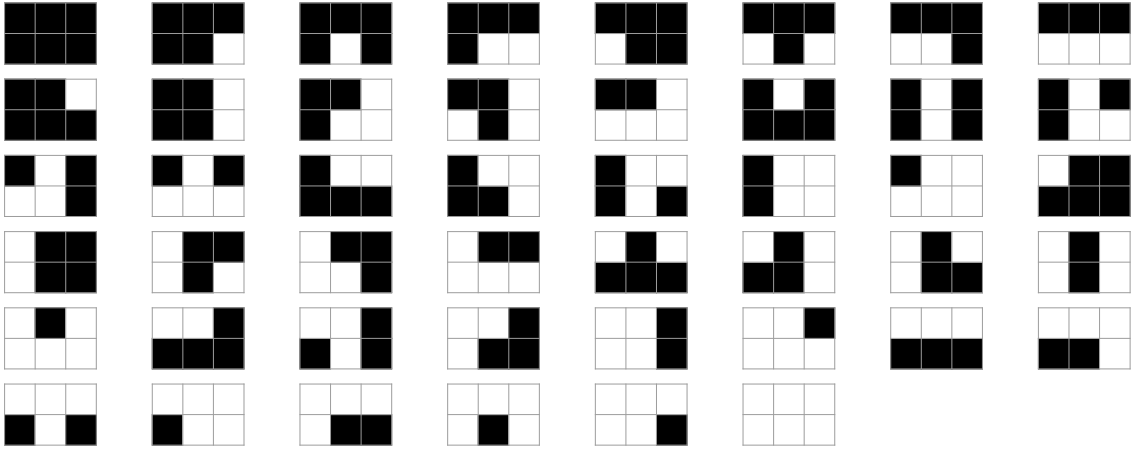}
\caption{Lonesum matrices of size $2\times 3$.} \label{fig1}
\end{figure}

Let $\Lo(n,k)$ be the set of lonesum matrices of size $n \times k$. Brewbaker \cite{Brewbaker} proved that
\begin{align*}
|\Lo(n,k)|=\sum_{m=0}^{\min\{n,k\}}m!{n+1 \brace m+1}m!{k+1 \brace m+1}.
\end{align*}

We define the following $q$-analogue of the \emph{poly-Bernoulli numbers}:
$$p_q(n,k):=\sum_{m=0}^{\min\{n,k\}}m!{n+1 \brace m+1}^*_qm!{k+1 \brace m+1}^*_q.$$

Given a matrix $A$, we define the set  $\Nu(A)$ as the indices of the all-$0$ columns and all-$0$ rows of $A$.   The \emph{weight $w_2$} of a lonesum matrix  $A$ is defined by letting
$$w_2(A):=q^{\sum_{i\in \Nu}i}.$$

For example,  for the matrix
$$A=
\left(
\begin{array}{ccccccccc}
 1 & 1 & 1 & 0 & 0 & 1 & 1 & 1 & 0 \\
 1 & 0 & 1 & 0 & 0 & 1 & 0 & 1 & 0 \\
 0 & 0 & 0 & 0 & 0 & 0 & 0 & 0 & 0 \\
 1 & 1 & 1 & 1 & 0 & 1 & 1 & 1 & 0 \\
 1 & 0 & 1 & 0 & 0 & 1 & 0 & 1 & 0 \\
 1 & 1 & 1 & 0 & 0 & 1 & 1 & 1 & 0 \\
\end{array}
\right)$$
we have that $\Nu(A)=\{3, 5, 9\}$. Therefore $w_2(A)=q^{17}.$

\begin{theorem}\label{teo1}
We have
$$p_q(n,k)=\sum_{A\in \Lo(n,k)}q^{w_2(A)}=\sum_{m=0}^{\min\{n,k\}}m!{n+1 \brace m+1}^*_qm!{k+1 \brace m+1}^*_q.$$
\end{theorem}

\begin{proof}
A lonesum matrix $A$ is determined by a pair of ordered partitions. See a detailed proof of how to construct bijectively a lonesum matrix from a given ordered partition pair for instance \cite{Benyi, BH1, BH, Brewbaker}. The key observation in a lonesum matrix is that two rows $R$ and $R'$ are either identical or the one with more $1$ entries, say $R'$, can be obtained from $R$ by switching some $0$s in $R$ to $1$s. The same is true for the columns. Hence, an ordered partition of the row indices corresponds to the order of the rows as follows: the indices of the identical rows are contained in the same block, and the order is determined by the number of $1$s in the particular rows. For the special all-zero rows,  we introduce a special block, by adding to the set of the row-indices the $0$ and taking the ordered partition of the extended set. Similarly, for columns. According to this procedure, the block containing $0$, contains the row indices of the all-zero rows. (The same is true for columns.) The theorem follows.
\end{proof}
\section{Inversions in Vesztergombi permutations}
In this section we describe a natural $q$-analogue of poly-Bernoulli numbers based on a permutation class. Vesztergombi \cite{Vesztergombi} investigated permutations with restrictions on the distance of the elements and their images (considering a permutation as a bijection on $\{1,2,\ldots, n\}$). Launois \cite{Launois} noticed that the number of such permutations with a particular bound is given by a poly-Bernoulli number. Several combinatorial proofs were later found for this fact \cite{Benyi, BH1, BH}. The $q$-analogue we introduce now arose in \cite{Sjostrand} during the study of intervals in Bruhat order. Here, we focus on the pure combinatorial point of view and on the theory of poly-Bernoulli numbers.
\begin{definition}
A permutation $\pi$ of $[n+k]$ is called \emph{Vesztergombi permutation} if
\[-k\leq \pi(i)-i \leq n
\]
for all $i\in[n+k]$.
\end{definition}
\begin{theorem}[\cite{Launois}]
Let $\mathcal{V}_n^k$ denote the set of Vesztergombi permutations. Then
\[|\mathcal{V}_n^k|=B_{n,k}.\]
\end{theorem}
For example,  $B_2^{(-2)}=14$, the permutation matrices corresponding to the Vesztergombi permutations of size $[2+2]$  appear in Figure \ref{fig2}. Note that  for a permutation $\pi \in \mathcal{V}_n^k$ this array is  obtained by placing a dot in row $i$ and column $\pi_i$ of an $(n+k) \times (n+k)$ array.
\begin{figure}[H]
\centering
\includegraphics[scale=0.45]{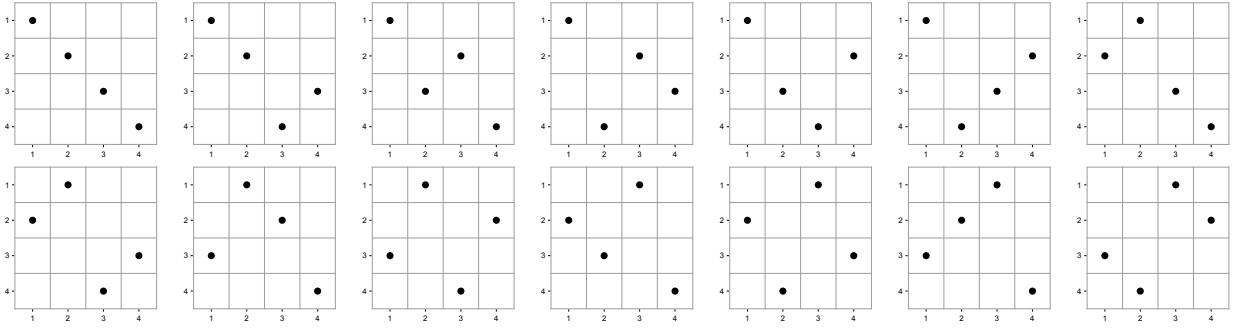}
\caption{Vesztergombi permutations of size $[2+2]$.} \label{fig2}
\end{figure}

Various proofs can be found in the literature: analytical \cite{Vesztergombi}, by constructing a bijection between lonesum matrices and Vesztergombi permutations \cite{Kim}, or using graph theoretical methods \cite{Benyi}. Here, we will use the description of the proof in \cite{BH}, since this is the one that can be used in the $q$-generalization.
For a permutation $\pi$, we denote by  $\mbox{inv}(\pi)$  the number of inversions of $\pi$, i.e., the number of pairs $(i, j)$, $i<j$ such that $\pi_i>\pi_j$. We define the $q$-analogue $pB_{n,k}(q)$ as follows:
\begin{align*}
pB_{n,k}(q):=\sum_{\pi\in \mathcal{V}_n^{k}}q^{\mbox{inv}(\pi)}.
\end{align*}

For example, from permutations in Figure \ref{fig2} we have that
$$pB_{2,2}(q)=1 + 3 q + 5 q^2 + 4 q^3 + q^4.$$

In Theorem \ref{teo7} we show a combinatorial identity for the sequence  $pB_{n,k}(q)$. Here we use the $q$-Stirling numbers $S_{n,k}(q)$ defined by the following recurrence
\begin{align*}
S_{n+1,k}(q)=q^{k-1}S_{n,k-1}+[k]_qS_{n,k}(q), \quad \mbox{for}\,0\leq k\leq n
\end{align*}
with the initial conditions $S_{0,0}(q)=1$ and $S_{n,k}(q)=0$ for $k<0$ or $k>n$.

Notice that these $q$-Stirling numbers have the following combinatorial interpretation
\begin{align*}
S_{n,k}(q)=\sum_{\pi \in \Pi_{n,k}}q^{\Inv^*(\pi)+\binom k2}.
\end{align*}

\begin{theorem}\label{teo7}
We have the combinatorial identity
\[pB_{n,k}(q)=q^{nk}\sum_{m=0}^{\min(n,k)}S_{n+1,m+1}(1/q)S_{k+1,m+1}(1/q)[m]!_q^2q^{m}.\]
\end{theorem}

One geometrical visualization of a Vesztergombi permutation can be given on a $(n+k)\times (n+k)$ matrix $V_{n+k}$ defined by:
\[
v_{ij}=\left\{\begin{array}{cl}
1, \quad& \mbox{if} \quad -k\leq i-j\leq n,\quad  i, j =1,\ldots,n+k;\\
0, \quad&\mbox{otherwise}.
\end{array}\right.
\]

The matrix $V$ is built up of 4 blocks, two all-1 matrices ($J_{n,k},J_{k,n}$), a lower ($T_n$) and an upper ($T^k$) triangular matrix. Precisely, let $V_{n+k}$ be the matrix:
\[
V_{n+k}=\left(
\begin{array}{cc}
J_{n,k} & T_n\\
T^k & J_{k,n}
\end{array}
\right),
\]
where $J_{n,k}\in\{0,1\}^{n\times k}$: $J_{n,k}(i,j)=1$ for all $i,j$,
$J_{k,n}\in\{0,1\}^{k\times n}$: $J_{k,n}(i,j)=1$ for all $i,j$,
$T_n\in\{0,1\}^{n\times n}$:
$T_n(i,j)=1$ if and only if $i\geq j$, and
$T^k\in\{0,1\}^{k\times k}$:
$T^k_{ij}=1$ if and only if $i\leq j$.
For a Vesztergombi permutation select exactly one $1$ entry from each row and column. This selection is actually one term in the permanent of the matrix $V_{n+k}$, hence, $|\mathcal{V}_{n}^k|$ is the permanent of the matrix $V_{n+k}$.

For example,  for $n=3$ and $k=2$ we have the matrix
$$V_{3+2}=V_5=\left(
\begin{array}{cc:ccc}
 1 & 1 & 1 & 0 & 0 \\
 1 & 1 & 1 & 1 & 0 \\
 1 & 1 & 1 & 1 & 1 \\ \hdashline
 1 & 1 & 1 & 1 & 1 \\
 0 & 1 & 1 & 1 & 1 \\
\end{array}
\right) \quad \text{and} \quad \text{perm}(V_5)=46.$$
Using other terms, we can say that this is a rook configuration on a board $V_{n+k}$, which is defined as the collection of the cells determined by the $1$ entries in the matrix.

Next, we recall some general results form the rook theory. A \emph{rook configuration}  $\mathcal{A}$  on a binary matrix $A$ are rooks on some of the $1$-entries of $A$ such that  no two rooks are in the same row or column.  The number of cells of $\mathcal{A}$ with no rook weakly to the right in the same row or below in the same column is denoted by $\mbox{inv}_A(\mathcal{A})$. In the special case where $A$ is an $n\times n$ matrix and $\mathcal{A}$ has $n$ rooks, i.e., the configuration corresponds to a permutation, $\pi$, by the representation $\pi_i=j$, $\mbox{inv}_A(\mathcal{A})$ became $\mbox{inv}(\pi)$.
The $k$th $q$-rook number for a board $\mathcal{A}$ is defined by Garsia and Remmel \cite{Garsia} as
\begin{align*}
R_{k}^{A}(q)=\sum_{\mathcal{A}}q^{\mbox{inv}_A(\mathcal{A})},
\end{align*}
where the sum is over all rook configurations on the matrix  $A$ with  $k$ rooks.

For example, in Figure \ref{Fig3} we show the Vesztergombi permutation  $\pi=31524\in V_3^2$ indicating how to calculate the inversions in this particular case. Notice that the rooks are denoted by black circles and the cells with circles denotes the positions  counted by the inversion statistic. The number of inversions of this rook configuration is 4, hence, its weight is $q^4$.
\begin{figure}[H]
\centering
\includegraphics[scale=0.8]{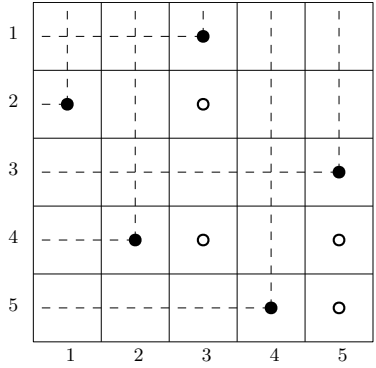}
\caption{Vesztergombi permutation  $\pi=31524\in V_3^2$ .} \label{Fig3}
\end{figure}
It is known \cite[Corollary 3]{Sjostrand}  that the $n$th $q$-rook number of the square matrix $J_{n,n}$ is given by $R_n^{J_{n,n}}(q)=[n]!_q$.  Further, Garcia and Remmel \cite[p. 248]{Garsia} showed that the $k$th $q$-rook number of the $n\times n$ matrix $H_n$ with ones on and above the secondary diagonal, that is
$$(H_n)_{i,j}=\begin{cases}
1,& i\leq n-j+1;\\
0,& \text{otherwise.}
\end{cases}$$
is given by
\begin{align*}
R_k^{H_n}(q)=q^{\binom{n}{2}}S_{n+1, n+1-k}(q).
\end{align*}
If we denote by $A'$ the matrix obtained by reflecting matrix $A$ upside down, then it is clear that $H_k=(T_k)'$. Further, from the relation $\mbox{inv}_{A'}(\mathcal{A})=\binom n2 - \mbox{inv}_A(\mathcal{A})$  we have that
\begin{align*}
R^{A'}_n(q)=q^{\binom{n}{2}}R_n^A\left(1/q\right).
\end{align*}
Finally, we need the result in \cite{Sjostrand} about the $(k+n)$th $q$-rook number of a matrix built up from the blocks of matrices $A$, $ B$, $J_{n,k}$ and $J_{k,n}$ as follows
\[
B/A:=\left(\begin{array}{cc}
B&J_{n,k}\\
J_{k,n}&A
\end{array}\right).
\]
For this case  the $(k+n)$th $q$-rook number  is
\begin{align*}
R_{k+n}^{B/A}(q)=\sum_{i=0}^{\min(k,n)}R_{k-i}^A(q)R_{n-i}^{B^*}(q)[i]!_qq^{-i^2},
\end{align*}
where $B^*$ is the matrix obtained by rotating the matrix $B$ in $180$ degrees.

It is clear that the matrix $V_{n+k}$ is the reflection of the matrix $H_k^*/H_n$.  Putting all the results above together, we have
\begin{align*}
pB_{n,k}(q)&=\sum_{\pi\in\mathcal{V}_n^k}q^{\mbox{inv}(\pi)}\\
&=R_{n+k}^{V_{n+k}}(q)=R_{n+k}^{(H_k^*/H_n)'}(q)=q^{\binom{n+k}{2}}R_{n+k}^{H_k^*/H_n}(1/q)\\
&=q^{\binom{n+k}{2}}\sum_{m=0}^{\min(n,k)}R_{n-m}^{H_n}(1/q)R_{k-m}^{H_k}(1/q)[m]!_{1/q}^2q^{m^2}\\
&=q^{\binom{n+k}{2}}\sum_{m=0}^{\min(n,k)}q^{-\binom{n}{2}}S_{n+1,m+1}(1/q)q^{-\binom{k}{2}}S_{k+1,m+1}(1/q)[m]!_{1/q}^2q^{m^2} \\
&=q^{nk}\sum_{m=0}^{\min(n,k)}S_{n+1,m+1}(1/q)S_{k+1,m+1}(1/q)[m]!_q^2q^{m}.
\end{align*}
The last equality follows from $[m]!_{1/q}=[m]!_qq^{-\binom{m}{2}}$.
\begin{remark}Sj\"{o}strand investigated  Coxeter groups in the paper \cite{Sjostrand}. He obtained the above formula as the Poincar\'{e} polynomial of the Bruhat interval $[\mbox{id},w]$ where $w$ is the maximal element in $A_{n-1}^{S\setminus \{s_k\}}$.
\[pB_{n,k}(q)=\mbox{Poin}_{[\mbox{id}, w]}(q).\]
\end{remark}

From the definition we have that
\begin{align*}
pB_{n,0}(q)&=pB_{0,n}(q)=1,\\
pB_{n,1}(q)&=pB_{1,n}(q)=(1+q)^n.\\
\end{align*}

Experimentally, we observed that for $n\geq2$,
\begin{align*}
pB_{n,2}(q)&=pB_{2,n}(q)=(1+q)W_n(-q),
\end{align*}
where $W_n(q)$ is the characteristic polynomial of the Sylvester  matrix $\mathcal{S}(P_{n}(q), P_{n+1}(q))$, with $P_n(q):=[n]_q=1+q+\cdots +q^{n-1}$.

For example,
 $$\mathcal{S}(P_{3}(q), P_{4}(q))=\left(
\begin{array}{ccccc}
 1 & 1 & 1 & 0 & 0 \\
 0 & 1 & 1 & 1 & 0 \\
 0 & 0 & 1 & 1 & 1 \\
 1 & 1 & 1 & 1 & 0 \\
 0 & 1 & 1 & 1 & 1 \\
\end{array}
\right)$$
and its characteristic polynomials is $W_3(q)=1 - 3 q + 6 q^2 - 7 q^3 + 5 q^4 - q^5$. Therefore,
\begin{align*}
pB_{3,2}(q)&=pB_{2,3}(q)=(1+q)W_3(-q)=1 + 4 q + 9 q^2 + 13 q^3 + 12 q^4 + 6 q^5 + q^6.
\end{align*}

\begin{question}
Prove (or disprove) this observation.
\end{question}

\section{$q$-analogue based on perm-matrices}
In this section we present a $q$-analogue based on a class of $01$-matrices avoiding a pair of $2\times 2$ submatrices. For the sake of getting ``nice'' results, we generalize this time the poly-Bernoulli relative, $C_n^{k}$,  analytically defined in \cite{AraKan} and combinatorially in \cite{BH} .

\begin{align*}
C_n^{k}=
\sum_{m=0}^{\min(n,k)}m!\stirling{n+1}{m+1} m!\stirling{k}{m},
\end{align*}

Our underlying objects are now \emph{perm-matrices}, a matrix class that can be defined by a forbiddance of two $2\times 2$ submatrices and with the extra condition that excludes all-$0$ columns.
\begin{definition}
A \emph{perm-matrix} is a $01$ matrix such that the following two properties hold:
\begin{enumerate}
\item each column contains at least one 1,
\item it avoids the submatrices
\end{enumerate}
\begin{align*}
\left(\begin{array}{cc}
   0&1\\
1&0
\end{array}\right)
\quad \mbox{and} \quad
\left(\begin{array}{cc}
   1&1\\
1&0
\end{array}\right).
\end{align*}

We denote the set of $n\times k$ perm-matrices by $\mathcal{P}_{n,k}$.
\end{definition}

It is proven for instance in \cite{BH} that
\begin{align*}
|\mathcal{P}_{n,k}|=C_n^k.
\end{align*}

The name is motivated by the fact that perm-matrices are special cases of permutation tableaux. Permutation tableaux were introduced by Steingr\'imsson and  Williams \cite{Stein-Wil}. They are a distinguished set of the Le-diagrams, introduced by Postnikov in the study of positive Grassmannians. Permutation tableaux are closely related to the partially asymmetric exclusion process, a model in statistical mechanics and are in bijection with permutations. A permutation tableau is defined as a $01$-filling of the cells of a Ferrers diagram satisfying the following two conditions:
\begin{enumerate}
\item each column contains a $1$,
\item there is no $0$ with a $1$ above in the same column and a $1$ to the left in the same row.
\end{enumerate}
We see that perm-matrices are permutation tableaux of rectangular shapes.
If we enumerate permutation tableaux with $k$ rows and $n-k$ columns according to the number of $1$s decreased by the number of columns, we obtain a $q$-analogue of the Eulerian numbers, $\widehat{E}_{n,k}(q)$, which polynomials were introduced and studied in \cite{Corteel2}. An explicit formula is
\begin{align*}
\widehat{E}_{n,k}(q)=q^{k-k^2}\sum_{i=0}^{k-1}(-1)^i[k-i]^nq^{ki-k}\left(\binom{n}{i}q^{k-i}+\binom{n}{i-1}\right).
\end{align*}
$\widehat{E}_{n,k}(q)$ specializes at $q=-1$ to binomial coefficients, at $q=0$ to Narayana numbers, and at $q=1$ to Eulerian numbers.

Here, we consider the special case, when the Ferrers diagram is a rectangle and define the $q$-analogue of the poly-Bernoulli numbers in a similar manner as the $q$-analogue of Eulerian numbers.

Given a matrix $M\in \mathcal{P}_{n,k}$ we define the weight of the matrix, $wt(M)$, as the total number of $1$s in the matrix reduced by the number of columns.

\begin{definition}
We introduce the $q$-analogue of the poly-Bernoulli relative as follows
\begin{align*}
C_n^k(q)=\sum_{M\in\mathcal{P}_{n,k}}q^{wt(M)}.
\end{align*}
\end{definition}

The motivation of the introduction of the $q$-analogue $C_n^k(q)$ comes from the strong connection to the theory of PASEP.  The partially asymmetric exclusion process, PASEP, is a model in statistical mechanics. The model describes the moves of particles on a one-dimensional lattice. Considering the lattice as an arrangements of $n$ cells, each cell can be occupied at the same moment by one particle and at each stage a particle can hop right or left. The model is asymmetric in the sense that the probability of hopping left is $q$ times the probability of hopping right. Additionally, particles may enter from left and exit to the right with given probability. (Here we assume always these probabilities to be $1$.) This model was studied intensively, we are here interested on the combinatorial aspects and connections, in particular on those results that are related to our $q$-poly-Bernoulli numbers.

According to Corteel and Williams \cite{Cortee1} the probability of finding the PASEP model such that exactly the first $k$ cells are occupied with particles in the steady state is
\begin{align*}
\frac{C_n^k(q)}{Z_{n+k}},
\end{align*}
where $Z_n$ is the partition function for the PASEP model, equal to the generating function for all permutation tableaux with $k$ rows and $n+1$ columns.

Moreover, according to \cite{Cortee1} $C_n^k(q)$ is the weight generating function of the following sets
\begin{itemize}
\item permutations such that the first $k$ elements are exactly the exceedances, enumerated according to crossings.
\item permutations such that the descents are the first $k$ elements, enumerated according to the occurrences of the pattern $2-31$.
\item weighted bicolored Motzkin paths such that the first $k$ steps are north and east steps, while the next $n$ steps are south and east steps, respectively.
\end{itemize}
For definitions and proofs see \cite{Cortee1}. The only difference here is that we require the \textit{first} $k$ elements to be exceedances and descents respectively,  and the \textit{first} $k$ steps to be north and east steps.
\begin{question}
Find explicit formulas for $C_{n}^k(q)$.
\end{question}

\section{$q$-analogue of Cenkci and Komatsu}

In this section we recall briefly some results of \cite{Cenkci} where the authors introduced and studied a $q$-analogue of poly-Bernoulli numbers.

The authors in \cite{Cenkci} defined poly-Bernoulli numbers with a $q$-parameter, $B_{n,q}^{(k)}$, for $n\geq 0$ and $k\geq 1$, $q$ a real number ($q\not=0$), by the generating function
\begin{align*}
\sum_{n=0}^{\infty}B_{n,q}^{(k)}\frac{t^n}{n!}=\frac{q\mbox{Li}_k\left(\frac{1-e^{-qt}}{q}\right)}{1-e^{-qt}}.
\end{align*}
The authors derived the analogue of the formulas of type \eqref{eq:pBszita} and \eqref{eq:rec}.  For example, we have
\begin{align*}
B_{n,q}^{(k)}&=\sum_{m=0}^{n}\stirling{n}{m}\frac{(-q)^{n-m}m!}{(m+1)^k},
\end{align*}
Further, the analogue of \eqref{eq:pBcomb} is true for all negative integer $k$:
\begin{align}\label{eq:qcomb}
B_{n,q}^{(k)}=q\sum_{j=0}^{\min(n,k)}(j!)^2S_2(n,j,q)S_2^{q^{-1}}(k+1,j+1),
\end{align}
where
\begin{align*}
S_2(n,j,q)&=\sum_{k=0}^{n}\binom{n}{k}q^{n-k}\stirling{k}{j}\quad \mbox{and}\\
\sum_{n=j}^{\infty}S_2^{q^{-1}}(n+1,j+1)\frac{t^n}{n!}&=\frac{(q^{-1}e^t-1)^jq^{-1}e^t}{j!}.
\end{align*}
We point out the similarity of the recursion \eqref{eq:rec} and the recursion that was derived for the so defined poly-Bernoulli numbers with a $q$ parameter in \cite{Cenkci}. We write the formula with negative $k$ indices in the form that emphasizes this similarity better.

\begin{align}\label{eq:qrec}
B_{n,q}^{(k+1)}=(n+1)B_{n,q}^{(k)}+ \sum_{i=1}^{n-1}q^{i}\binom{n}{i+1}B_{n-i,q}^{(k)}.
\end{align}
For the original formula \eqref{eq:rec} $\Gamma$-free matrices give a transparent explanation.
\emph{$\Gamma$-free matrices} were introduced in \cite{BH} as $01$ matrices avoiding the following two $2\times 2$ submatrices
\[
\Gamma =\left\{
 \begin{pmatrix} 1 & 1\\1 & 0 \end{pmatrix},
\begin{pmatrix}1 & 1 \\ 1 & 1\end{pmatrix}\right\}.
\]
For example,  the 46  $\Gamma$-free matrices  of size $2\times 3$ appear in Figure \ref{fig4}. Note that we use  black squares for the ones and white squares for the zeros.
\begin{figure}[H]
\centering
\includegraphics[scale=1]{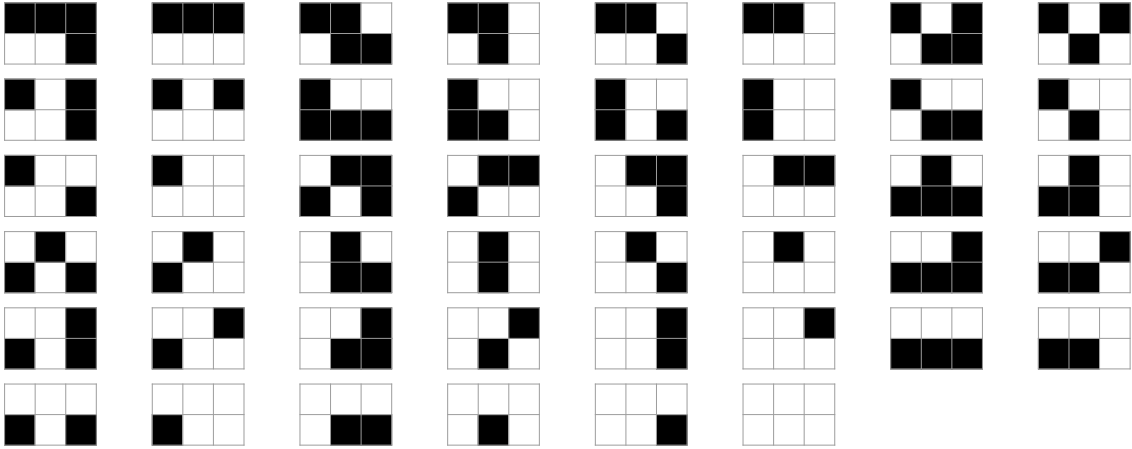}
\caption{$\Gamma$-free  matrices of size $2\times 3$.} \label{fig4}
\end{figure}

A bijective proof of the fact that $n\times k$ $\Gamma$-free matrices are enumerated by the poly-Bernoulli numbers, $B_n^{(-k)}$ was given in \cite{BH} and \cite{Nagy}. In \cite{BH} the authors give the explanation of the recursion \eqref{eq:rec}, which might be used to give the explanation for \eqref{eq:qrec}.

The main observation in a $\Gamma$-free matrix is that if a column contains at least two $1$s, the rows of these $1$s are contained include only $0$s to the right of these $1$s, except the row of the bottom most $1$. So, a $\Gamma$-free matrix with $k+1$ columns can be built as follows: choose some row indices, $\mathcal{R}$, write into these rows in the first column a $1$, and fill the rest of these rows with $0$s, except the bottom most row, the greatest element in $\mathcal{R}$. For the rest of the matrix there are no further restrictions, it can be constructed as an arbitrary $\Gamma$-free matrix with $n-|\mathcal{R}|+1$ rows and $k$ columns. Let $|\mathcal{R}|=i+1$, so that $i$ is the number of rows that we filled up with $0$ to the right of $1$ in the first column. By iterating this observation, we see that $q$ marks exactly those $1$s that have only $0$s to their right.

\begin{question}
Is it possible to give a combinatorial interpretation of $B_{n,q}^{(-k)}$ using $\Gamma$-free matrices or related combinatorial objects?
\end{question}

\section{Akiyama-Tanigawa algorithm}
In this section we give a generalization of the formula \eqref{eq:pBszita} replacing each part of the expression by its $q$-analogue. We use again  the $q$-Stirling numbers $\stirling{n}{m}_q$ introduced by Carlitz \cite{Carlitz}. Remember that they are defined by the recurrence relation
\[
\stirling{n}{m}_q=\stirling{n-1}{m-1}_q+[k]_q\stirling{n-1}{m}_q,
\]
with $\stirling{n}{0}_q=\stirling{0}{m}_q=0$ except $\stirling{0}{0}_q=1$.

\begin{definition}
We define the $q$-analogue of poly-Bernoulli number for any  $k\in \Z$ and $n\geq 0$ as
\begin{align*}
p_{n,k}(q):=(-1)^n\sum_{m=0}^{n}\frac{(-1)^m [m]!}{[m+1]^k}\stirling{n}{m}_q.
\end{align*}

\end{definition}

We show that this definition naturally arise from the $q$ generalization of the  Akiyama-Tanigawa algorithm defined by Zeng \cite{Zeng}. The Akiyama Tanigawa algorithm generates the Bernoulli numbers in a similar manner as Pascal's triangle the binomial coefficients. Akiyama and Tanigawa found this amusing algorithm during investigations of the multiple zeta functions \cite{Akiyama}. It is actually very similar to the Euler-Seidel matrix \cite{Zeng}, and strongly connected to Riordan arrays \cite{Merlini}.
The algorithm reads as follows, start with the $0$-th row $1, \frac{1}{2}, \frac{1}{3}, \ldots$ Define the first row as $1\cdot\left(1-\frac{1}{2}\right),  2\cdot\left(\frac{1}{2}-\frac{1}{3}\right), 3\cdot\left(\frac{1}{3}-\frac{1}{4}\right), \ldots$, giving the sequence $\frac{1}{2}, \frac{1}{3}, \frac{1}{4}, \ldots$. The next row is produced by $1\cdot\left(\frac{1}{2}-\frac{1}{3}\right), 2\cdot\left(\frac{1}{3}-\frac{1}{4}\right), 3\cdot\left(\frac{1}{4}-\frac{1}{5}\right), \ldots$. So the second row is $\frac{1}{6}, \frac{1}{6}, \frac{3}{20}, \ldots$. Generally, let $a_{k,n}$ denote the $k$th number in the $n$th row. Then $a_{n+1,k}$ is defined by the following recurrence:
\begin{align*}
a_{n+1,k}=(k+1)(a_{n,k}-a_{n,k+1})
\end{align*}
As a result of this generating rule, the sequence of the leading numbers of the rows, $a_{n,0}$ is the sequence of the Bernoulli numbers.  Kaneko \cite{KanekoAkiyama} showed that if the initial sequence is $1$, $\frac{1}{2^k}$, $\frac{1}{3^k}$, $\ldots$ instead of $1$, $\frac{1}{2}$, $\frac{1}{3}$, $\ldots$ the resulting sequence gives the poly-Bernoulli numbers instead of the Bernoulli numbers.

The $q$-analogue of this algorithm was investigated by Zeng \cite{Zeng} in order to define a similar computing rule for the $q$-Bernoulli numbers $\beta_n$ introduced by Carlitz as
\begin{align*}
q(q\beta +1)^n-\beta_n=\left\{\begin{array}{ccc}
    1,& \mbox{if} & n=1;\\
0, & \mbox{if}& n>1;
\end{array}\right.
\end{align*}
where $\beta_0=1$ and $\beta_k=\beta^k$ after expansion.

The key theorem of this study is the following.

\begin{theorem}\cite{Zeng}\label{th:Zeng}
Given an initial sequence $a_{0,m}$ ($m\geq 0$), let $a_{n+1,m}$ be defined recursively by
\begin{align*}
A.\quad  a_{n+1,m}&=[m+1](a_{n,m}-a_{n,m+1}),\\
B.\quad a_{n+1,m}&=[m]a_{n,m}-[m+1]a_{n,m+1}.
\end{align*}
Then
\begin{align*}
A.\quad a_{n,0}&=\sum_{m=0}^n(-1)^m[m]! \stirling{n+1}{m+1}_qa_{0,m},\\
B.\quad a_{n,0}&=\sum_{m=0}^n(-1)^m[m]! \stirling{n}{m}_qa_{0,m}.
\end{align*}
\end{theorem}

Zeng showed that if the $0$th row, i.e., the initial sequence $a_{0,m}$ is $1$, $\frac{1}{[2]_q}$, $\frac{1}{[3]_q}$, $\ldots$ then the sequence $a_{n,0}$ are the $q$-Bernoulli numbers $\beta_n$ and presented a new formula for $\beta_n$, since the theorem implies:
\begin{align*}
\beta_n=\sum_{k=0}^{n}(-1)^k\stirling{n+1}{k+1}_q \frac{[k]!_q}{[k+1]_q}\quad (n\geq 2).
\end{align*}

Turning now to the generalization of the poly-Bernoulli numbers, the next theorem is a straightforward consequence of these previous results.

\begin{theorem}
Given the sequence $a_{0,m}=[m+1]^k$, $(m\geq 0)$, let $a_{n+1,m}$ be defined recursively by
\[ a_{n+1,m}=[m]a_{n,m}-[m+1]a_{n,m+1}.\]
Then the sequence which are built from the first element of each row is the $q$-poly-Bernoulli numbers.
\[a_{n,0}=p_{n,k}(q).\]
\end{theorem}
Next, we give an expression for the generating function of $p_{n,k}(q)$:
\begin{align*}
\sum_{n=0}^{\infty}p_{n,k}(q)\frac{z^n}{[n]!}
\end{align*}

Ernst gives (see Theorem 5.2.17 in \cite{Ernst}) the following generating function for the $q$-Stirling numbers
\begin{align}\label{qgefun}
\sum_{n=m}^{\infty} \stirling{n}{m}_q\frac{z^n}{[n]!}=\frac{1}{[m]!q^{\binom m2}}\sum_{i=0}^m  {m \brack i}_q(-1)^iq^{\binom i2} E_q(z[m-i]),
\end{align}
where the $q$-binomial is defined by $ {n \brack k}_q= \frac{ [n]_q!}{ [k]_q! [n-k]_q!},$ and the $q$-exponential is given by
$$E_q(z):=\sum_{k=0}^{\infty}\frac{z^k}{[k]!}.$$

Note that for $q=1$ we recover the well-known generating function for the Stirling numbers of the second kind
$$\sum_{n=m}^{\infty} \stirling{n}{m}\frac{z^n}{n!}=\frac{(e^{z}-1)^m}{m!}.$$

From \eqref{qgefun} we obtain the following expression
$$\sum_{n=0}^{\infty}p_{n,k}(q)\frac{z^n}{[n]!}=\sum_{i=0}^{\infty}\sum_{m=0}^{\infty}\frac{(-1)^m}{[m+i+1]^k}q^{\binom i2- \binom{m+i}{2}}{m+i \brack i}_q E_q(-z[m]).$$
For $q=1$ we recover \eqref{gefunpoly}.

\begin{question}
Is there any combinatorial interpretation of the $q$-poly-Bernoulli numbers $p_{n,k}(q)$?
\end{question}

\end{document}